\documentclass[pamq, nonatbib]{ipart}

\RequirePackage{hyperref}

\startlocaldefs

\usepackage{stmaryrd}

\usepackage{tikz-cd}
\usepackage{graphicx}
\usepackage[colorinlistoftodos]{todonotes}
\usepackage{extarrows}
\usepackage{parskip}
\usepackage{bm}
\usepackage{soul}
\usepackage{url}

\newtheorem{thm}[subsection]{Theorem}
\newtheorem{prop}[subsection]{Proposition}
\newtheorem{lemma}[subsection]{Lemma}
\newtheorem{cor}[subsection]{Corollary}

\theoremstyle{definition}
\newtheorem{defn}[subsection]{Definition}
\newtheorem{hypo}[subsection]{Hypothesis}
\newtheorem{ex}[subsection]{Example}
\theoremstyle{remark}
\newtheorem{rmk}[subsection]{Remark}
\numberwithin{equation}{section}

\newtheoremstyle{named}{}{}{\itshape}{}{\bfseries}{.}{.5em}{#1 \thmnote{#3}}
\theoremstyle{named}

\newtheoremstyle{named}{}{}{\itshape}{}{\bfseries}{.}{.5em}{#1 \thmnote{#3}}
\theoremstyle{named}

\newtheoremstyle{named}{}{}{\itshape}{}{\bfseries}{.}{.5em}{#1 \thmnote{#3}}
\theoremstyle{named}

\makeatletter
\newcommand{\mylabel}[2]{#2\def\@currentlabel{#2}\label{#1}}
\makeatother

\makeatletter
\newcommand{\leqnomode}{\tagsleft@true}
\newcommand{\reqnomode}{\tagsleft@false}
\makeatother

\newcommand{\Q}{\mathbb{Q}}
\newcommand{\Z}{\mathbb{Z}}

\newcommand{\C}{\mathbb{C}}

\newcommand{\F}{\mathbb{F}}

\newcommand{\calO}{\mathcal{O}}
\newcommand{\calS}{\mathcal{S}}
\newcommand{\Sel}{\text{Sel}}
\newcommand{\corank}{\text{corank}}

\newcommand{\cyc}{\text{cyc}}

\newcommand{\Aut}{\text{Aut}}
\newcommand{\pd}{\text{pd}}

\DeclareMathOperator{\Ima}{Im}
\DeclareMathOperator{\Gal}{Gal}

\newcommand{\frakp}{\mathfrak{p}}

\newcommand{\frakm}{\mathfrak{m}}

\newcommand{\frakN}{\mathfrak{N}}
\newcommand{\fraka}{\mathfrak{a}}

\newcommand{\Pic}{\text{Pic}}

\DeclareMathOperator{\Hom}{\text{Hom}}

\DeclareMathOperator{\disc}{\text{disc}}

\DeclareMathOperator{\Cores}{Cores}
\DeclareFontFamily{U}{wncy}{}
\DeclareFontShape{U}{wncy}{m}{n}{<->wncyr10}{}
\DeclareSymbolFont{mcy}{U}{wncy}{m}{n}
\DeclareMathSymbol{\Sha}{\mathord}{mcy}{"58}

\newcommand{\ac}{\text{ac}}

\newcommand{\Tor}{\text{Tor}}
\newcommand{\coker}{\text{coker}}

\newcommand{\calL}{\mathcal{L}}

\newcommand{\vbar}{\overline{v}}
\newcommand{\calH}{\mathcal{H}}
\newcommand{\depth}{\text{depth}}
\endlocaldefs

\pubyear{2022}
\volume{0}
\issue{0}
\firstpage{1}
\lastpage{1}
\begin{document}
	
	\begin{frontmatter}
		
		%%  "Title of the Paper"
		\title{Galois cohomology of elliptic curves over anticyclotomic extensions}
		%\thankstext{T1}{???}
		
		\begin{aug}
			%%  \author{\fnms{John} \snm{Smith}\thanksref{t2}\ead[label=e1]{smith@foo.com}\ead[label=e2,url]{www.foo.com}}
			%%  \thankstext{t2}{The author is supported by ...}
			%%  \address{line 1\\ line 2\\ \printead{e1}\\\printead{e2}}
			
			\author{\fnms{Dac-Nhan-Tam} \snm{Nguyen} %\thanksref{t1}
				\ead[label=e1]{tamnguyen@math.ubc.ca}}
			\address{Department of Mathematics, University of British Columbia \\
				1984 Mathematics Road, Vancouver, BC V6T 1Z2, Canada \\\printead{e1}}
			%\author{\fnms{???} \snm{???}\thanksref{t2}\ead[label=e2]{???}}
			%\address{???\\\printead{e2}}
			%\and
			
			\author{\fnms{Sujatha} \snm{Ramdorai} %\thanksref{t2}
				\ead[label=e2]{sujatha@math.ubc.ca}}
			\address{Department of Mathematics, University of British Columbia \\
				1984 Mathematics Road, Vancouver, BC V6T 1Z2, Canada \\\printead{e2}}
			%\author{\fnms{???} \snm{???}%
				%        \ead[label=e3]{???}%
				%        \ead[label=u1,url]{???}}
			%\address{???\\\printead{e3}\\\printead{u1}}
			%\thankstext{t1}{University of British Columbia}
			%\thankstext{t2}{The author is supported by ...}
		\end{aug}
		%%  History:
		%\received{\sday{3} \smonth{1} \syear{2022}}
		
		\begin{center}
			\textit{Dedicated to the Memory of John Coates}
			%\\[1cm]
		\end{center}
		
		\begin{abstract}
			Let $K$ be an imaginary quadratic field and $p$ be an odd prime number.
			Let $E/\Q$ be an elliptic curve with good ordinary reduction at $p$. We study the Iwasawa theory of $E$ over the anticyclotomic $\Z_p$-extension of $K$ by adopting a unifying framework.
			We also study the Galois cohomology of the dual Selmer group of $E$ over the unique $\Z_p^2$-extension
			of $K$ as well over the anticyclotomic extension of $K$.
		\end{abstract}
		
		
		%%  Upper case for every keyword
		
		%\tableofcontents
		
	\end{frontmatter}
	
	%%  The body
	
	\section{Introduction}
	
	Let $K$ be an imaginary quadratic number field and $p \geq 3$ be an odd prime number. By Leopoldt's conjecture, which is 
	known to hold over $K$, there are two linearly independent $\Z_p$-extensions, namely the cyclotomic $\Z_p$-extension
	and the anticyclotomic ${\Bbb Z}_p$-extension, which we denote by $K_\cyc$ and $K_{\ac}$ respectively. Let $E$ be an elliptic curve over 
	$K$  which is ordinary at the primes above $p.$ Iwasawa theory of elliptic curves over cyclotomic ${\mathbb Z}_p$-extensions has been  studied for almost five decades,
	and that over the  anticyclotomic extension spans a  period of smaller duration. The results are significantly different over these two extensions.
	While a conjecture of Mazur, later proved by Kato, asserts that the dual Selmer group  $X(E/K_\cyc)$ over the cyclotomic extension
	is torsion as a module over the corresponding Iwasawa algebra, the dual Selmer group  $X(E/K_{\ac})$ is not necessarily torsion as 
	a module over the corresponding Iwasawa algebra and may have positive rank. When the rank is positive, the contribution to the rank is related to the 
	existence of Heegner points in the Mordell-Weil group of the elliptic curve along the anticyclotomic extension. Iwasawa  theory of $E$ in the anticyclotomic setting has
	been studied by Perrin-Riou \cite{P-R84, P-R87, P-R92, P-R95} and later by Howard \cite{How04, How04b}, Bertolini-Darmon \cite{BD05}, Agboola-Howard \cite{AH06} depending on the different contexts that are possible in this setting. 
	
	The goal in this article is to approach the Iwasawa theoretic study of elliptic curves over anticyclotomic ${\mathbb Z}_p$-extensions from
	within a unified framework that is largely algebraic. In the case that the dual Selmer group is torsion over the anticyclotomic extension, we show that
	the methods and results are analogous to those in  the cyclotomic extension. If the dual Selmer group over the anticyclotomic extension has positive rank,
	then we recover most of the known results by far simpler methods. Our approach and techniques have the advantage that they can be used to illustrate the 
	main results with concrete numerical examples.
	
	This article consists of five sections, including this introductory section. Section \ref{sec:prelim} is preliminary in nature. In Section \ref{sec:literature}, we recall results in the literature regarding the module structure of the Selmer group over the anticyclotomic extension. In both Section \ref{sec:def} and Section \ref{sec:indef}, we study the Galois cohomology of Selmer groups in the generic case. In Section \ref{sec:def}, we study the module structure of the Selmer group over the compositum $K_\cyc \cdot K_\ac$. In Section \ref{sec:indef}, we recover the results in \cite{Ber01} with different but simpler hypotheses. In Section \ref{sec:mu=0}, we give an overview of the literature on the vanishing of the $\mu$-invariant and again obtain some results using simpler methods.
	
	It is a privilege and pleasure to dedicate this article in memory of John Coates. While the influence and contributions of John Coates in the study of
	Iwasawa theory is well-known and recognised, we would like to mention that he also deeply felt the lack of articles that would consolidate the
	scattered results in the literature. The monograph \cite{CS00} was written largely with this view, and he always felt that a similar endeavour should be 
	made to present the results in Iwasawa theory of elliptic curves in the anticyclotomic setting. This article  is a humble attempt in that direction.

	\section{Preliminaries} \label{sec:prelim}
	
	\subsection{Selmer groups}
	
	Let $p$ be an odd prime. Let $K$ be a number field and $E/K$ be an elliptic curve which is ordinary at every prime above $p$. Denote by $S$  a finite set of primes containing the primes above $p$ and the primes of bad reduction of $E/K$. Let $K_S$ be the maximal  extension that is unramified outside of $S$. 
	
	For every Galois extension $L/K$ contained in $K_S$ and a prime $v$ in $K$, we will define the group $J_v(E/L)$ as
	\[J_v(E/L) = \bigoplus_{w \mid v} H^1(L_w, E)(p)\]
	where the direct sum is over the primes $w$ of $L$ lying above $v$. For a finite extension $L/K$, we define the $p^n$-Selmer group $\Sel(E_{p^n}/L)$ as the kernel of the natural restriction map
	\[H^1(K_S/L, E_{p^n}) \rightarrow \bigoplus_{v \in S} J_v(E/L),\] 
	and define the $p$-primary Selmer group $\Sel_p(E/L)$ as the direct limit \[\Sel_p(E/L) := \varinjlim_{n} \Sel(E_{p^n}/L).\] For a general Galois extension $L/K$, we define $\Sel_p(E/L) = \varinjlim_{F} \Sel_p(E/F)$ where $F$ runs over the finite extensions of $K$ contained in $L$. If $L$ is contained in $K_S$, then there is an exact sequence
	\[0 \rightarrow \Sel_p(E/L) \rightarrow H^1(K_S/L, E_{p^\infty}) \rightarrow \bigoplus_{v \in S} J_v(E/L).\]
	For a finite Galois extension $L/K$, we define the compact Selmer group as the projective limit $\calS_p(E/L) := \varprojlim_{n} \Sel_p(E_{p^n}/ F)$ with respect to the projections $E_{p^{n + 1}} \xrightarrow{\times p} E_{p^n}$. For a general Galois extension $L/K$, let $\calS_p(E/L)$ be the inverse limit $\calS_p(E/L) := \varprojlim_{F} \calS_p(E/F)$ with respect to corestriction maps, where $F$ runs over the finite extensions of $K$ contained in $L$.
	
	\subsection{Modules over the Iwasawa algebra}
	
	Let $\calL/K$ be a $p$-adic Lie extension and let $G = \Gal(\calL/K)$. Denote by $\Lambda(G)$  the Iwasawa algebra
	\[\Z_p \llbracket \Gal(\calL/K) \rrbracket := \varprojlim_{F \subset \calL, F/K \text{ finite}} \Z_p [\Gal(F/K)].\]
	For a $\Z_p$-module $A$, we will denote by \[A^\vee = \Hom(A, \Q_p/\Z_p)\] its Pontryagin dual. The $p$-primary group $\Sel_p(E/\calL)$ has the discrete topology as a $\Lambda(G)$-module and its Pontryagin dual $X(E/\calL) := \Sel_p(E/\calL)^\vee$ is a compact $\Lambda(G)$-module. 
	
	When $G \simeq \Z_p$ and $X$ is a finitely-generated $\Lambda(G)$-module, we will denote the Iwasawa $\mu$ and $\lambda$-invariants of $X$ by $\mu_G(X)$ and $\lambda_G(X)$, respectively. The group $G$ is often clear from the context, in which case we will suppress $G$ from the notation and let $\mu(X) := \mu_G(X), \lambda(X) := \lambda_G(X)$.
	
	\section{Module structure of Selmer groups over the anticyclotomic extension} \label{sec:literature}
	Let $E/\Q$ be an elliptic curve of conductor $N$, $p \geq 5$ be a prime where $E$ has good ordinary reduction. Let $K$ be an imaginary quadratic field of discriminant $\disc(K) \neq -3, -4$. We may decompose $N = N^{+}N^{-}$ where $N^{+}$ is the product of the primes dividing $N$ that split in $K/\Q$ and $N^{-}$ is the product of the primes that are inert. Following the terminology in \cite{Ber01}, we say that $E$ is in the generic case if $E$ does not have complex multiplication (CM) or the field of CM of $E$ is different from $K$. Otherwise, we say that $E$ is in the exceptional case.
	
	Recall that we denote by $K_\cyc$ and $K_\ac$ the cyclotomic and anticyclotomic extensions over $K$, respectively. Moreover, let $K_\infty = K_\ac \cdot K_\cyc$ be their compositum and let $G_\infty = \Gal(K_\infty/K)$. The Galois groups $\Gamma_\cyc, H_\cyc, \Gamma_\ac, H_\ac$ are defined as follows:
	\[\Gamma_\cyc = \Gal(K_\cyc/K), H_\cyc = \Gal(K_\infty/K_\cyc),\]
	\[\Gamma_\ac = \Gal(K_\ac/K), H_\ac = \Gal(K_\infty/K_\ac).\]
	
	\subsection{Generic case}
	
	First consider the case when the complex $L$-function $L(E/K, s)$ vanishes to even order at $s = 1$. This occurs, for example, when $N^{-}$ is a square-free product of an odd number of primes. In this case, it is due to Bertolini-Darmon that $X(E/K_\ac)$ is $\Lambda(\Gamma_\ac)$-torsion under mild technical hypotheses \cite{BD05}. This is also known as the {\it definite} setting in the literature. On the other hand, in the case where $N^{-}$ is the product of an even number of primes, the module $X(E/K_\ac)$ has $\Lambda(\Gamma_\ac)$-rank $1$ \cite{How04, How04b}. This is known as the {\it indefinite} setting. The reasoning behind such terminology is related to the quaternion algebra over $\Q$ whose discriminant has finite part equal to $N^{-}$ and the study of its associated Shimura variety. For further details on this terminology, readers are advised to consult \cite[Section 2]{BLV23}.
	
	Let us first consider the definite setting. Because $X(E/K_\ac)$ is $\Lambda(\Gamma_\ac)$-torsion, the result of Greenberg \cite[Section 7]{Gre89} allows us to conclude that $X(E/K_\ac)$ does not have any non-trivial finite $\Lambda(\Gamma_\ac)$-submodules. Greenberg's result concerns Selmer groups of general $p$-adic Galois representations and the statement below is a simplified version for elliptic curves due to Bertolini.
	
	\begin{thm} \cite[Theorem 4.5]{Ber01} \label{thm:Ber-tor}
		Let $\mathcal{T}$ be the order of the torsion group of $\oplus_{v \in S \setminus \{p\}} E(K_\ell)$. Assume that 
		\begin{enumerate}
			\item $p \nmid 6 \mathcal{T},$
			\item $X(E/K_\ac)$ is a $\Lambda(\Gamma_\ac)$-torsion module.
		\end{enumerate}
		Then $X(E/K_\ac)$ has no non-trivial finite $\Lambda(\Gamma_\ac)$-submodules.
	\end{thm}

	It should be noted that this statement applies more generally to any $\Z_p$-extension, and the key hypothesis is that $X(E/K_\ac)$ is a $\Lambda(\Gamma_\ac)$-torsion module. 
	
	When $E(K)$ is finite, one can readily deduce that $X(E/K_\ac)$ is $\Lambda(\Gamma_\ac)$-torsion, by using the well-known fact that the control morphism $\Sel_p(E/K) \rightarrow \Sel_p(E/K_\ac)^{\Gamma_\ac}$ is always a pseudo-isomorphism. In this case, the argument of Theorem 3.11 in the work of Coates and the second author \cite{CS00} in the cyclotomic setting can be applied to give an alternative proof of the non-existence of non-trivial finite $\Lambda(\Gamma_\ac)$-submodules. 
	
	We first recall some notions and results in commutative algebra. For a ring $R$ and an $R$-module $M$, let $\pd_{R}(M)$ denote the projective dimension of $M$ as an $R$-module.
	
	\begin{prop} \cite[Proposition 3.10]{Ven02} \label{prop:Ven}
		Let $G$ be a $p$-adic Lie group. Then a compact $\Lambda(G)$-module $M$ does not have any non-trivial pseudo-null $\Lambda(G)$-submodules if $\pd_{\Lambda(G)}(M) \leq 1$.
	\end{prop}
	
	We recall the definition of a regular sequence and the depth of a module $M$ over a ring $R$, following \cite{Eis95}.
	
	\begin{defn}\label{def:regular sequence}
		Suppose $ R $ is a ring and let $ M $ be an $ R $-module. 
		A sequence of elements $ f_1,\cdots,f_k $ of $ R $ is called an $ M $-regular sequence if the following conditions hold:
		\begin{enumerate}
			\item The element $ f_i $ is a non-zero divisor on $ M/(f_1,\cdots,f_{i-1})M $ for each $ i=1,\cdots,k $;
			\item the module $ M/(f_1,\cdots,f_k)M $ is not zero.	
		\end{enumerate}
		If $ I $ is an ideal of the ring $ R $ and $ f_1,\cdots,f_k \in I $ then we call $ f_{1},…,f_{k} $ an $ M $-regular sequence in $ I $. 
		Moreover, if $ M=R $, we call $ f_{1},\cdots,f_{k} $ a regular sequence in $ I $.
	\end{defn}
	\begin{defn}\label{def:depth} %(Definition 5.1 of \cite{venjakob2002structure} and Section 18.1 of \cite{Eisenbud1995})
		Let $ I $ be an ideal of a ring $ R $ and suppose $ M $ is a finitely generated $ R $-module such that $ IM \neq M $. 
		Then the $ I $-depth of $ M $, denoted by $ \depth_I(M) $, is the maximal length of a $ M $-regular sequence in $ I $. 
		When $R$ is a local ring, the depth of $ M $ as an $R$-module, denoted by $ \depth(M) $, is defined to be
		\[
		\depth(M) := \depth_\frakm(M)
		\]
		where $ \frakm $ is the maximal ideal of $ R $.
	\end{defn}
	
	When $R$ is a commutative Noetherian regular local ring, $\depth(R)$ coincides with the Krull dimension $\dim(R)$. In particular, when $G \simeq \Z_p^n$ for some positive integer $n$ we have
	\[\depth(\Lambda(G)) = \dim(\Lambda(G)) = n + 1.\]
	
	\begin{prop} \label{lem:H1=0}
		Let $\Gamma \simeq \Z_p$ and $M$ be a compact $\Lambda(\Gamma)$-module such that $H_1(\Gamma, M) = 0$. Then $M$ does not have any non-trivial finite submodules.
	\end{prop}
	
	\begin{proof}
		Fix an isomorphism $\Lambda(\Gamma) \simeq \Z_p \llbracket T \rrbracket$. Let $M[T] := \{m \in M \mid Tm = 0\}.$ Since $\Lambda(\Gamma)/T \Lambda(\Gamma) \simeq \Z_p$, it follows from definition and basic properties of $H_1(\Gamma, -)$
		that $H_1(\Gamma, M) = \Tor_1^{\Lambda(\Gamma)}(M, \Z_p) = M[T]$ and thus $M[T] = 0$ under our hypothesis.
		
		Moreover, if $M/TM = 0$ then $M/(T, p) M = 0$, which implies $M = 0$ by Nakayama's lemma for the local ring $\Lambda(\Gamma)$. Therefore, $M/TM \neq 0$ 
		and $T$ is an $M$-regular sequence. 
		Thus $\depth(M) \geq 1$, which means $\pd_{\Lambda(\Gamma)} M \leq 1$ using the Auslander-Buchsbaum formula \cite[Theorem 19.9]{Eis95}:
		\[\pd_{\Lambda(\Gamma)} (M) + \depth(M) = \dim(\Lambda(\Gamma)) = 2.\]
		The result now follows from Proposition \ref{prop:Ven}. An alternative proof is also given in \cite[Lemma A.1.5]{CS00}. 
	\end{proof}
	
	Using Proposition \ref{lem:H1=0}, we obtain
	\begin{thm} \cite[Theorem 3.11]{CS00} \label{thm:tor}
		Assume that \begin{enumerate}
			\item $p > 2$,
			\item $\Sel_p(E/K)$ is finite,
			\item For each $P \neq 0$ in $E(K)(p)$, there exists a place $v$ in $K$ dividing $p$ such that the reduction of $P$ modulo $v$ is nonzero.
		\end{enumerate}
		Then $H^1(\Gamma_\ac, \Sel_p(E/K_\ac)) = 0$, and the module $X(E/K_\ac)$ does not have any non-trivial finite $\Lambda(\Gamma_\ac)$-submodules.
	\end{thm}

	The argument of Theorem \ref{thm:Ber-tor} no longer holds in the indefinite setting because $X(E/K_\ac)$ is not $\Lambda(\Gamma_\ac)$-torsion. Nonetheless, Bertolini was able to show that $X(E/K_\ac)$ does not have any non-trivial finite $\Lambda(\Gamma_\ac)$-submodules under certain hypotheses. 
	
	For a prime $v$ in $K$ above $p$, denote by $\widetilde{E}_{v}$ the reduction of $E/K$ modulo $v$ and let $k_v$ be the residue field of $K_v$. Moreover, let $d_v = \# \widetilde{E}_v(k_v)$. For each prime $\ell \mid N$, denote by $c_\ell(E/\Q)$ the Tamagawa factor of $E/\Q$ at $\ell$. For an integer $n$, let $n^{(p)}$ be the largest prime power that divides $n$.
	
	\begin{thm} \cite[Theorem 7.1]{Ber01} \label{thm:Ber2}
		Assume that \begin{enumerate}
			\item $p \nmid 6 N \disc(K) \#\Pic(\calO_K),$
			\item $\overline{\rho}_{E, p}: G_\Q \rightarrow Aut(E_p)$ is surjective,
			\item $p \nmid \prod_{v \mid p} \# \widetilde{E}_v(k_v) \cdot \prod_{\ell \mid N} c_\ell(E/\Q),$
			\item The module of Heegner points over $K_n$ is non-zero modulo $pE(K_n)$ for some $n \geq 0$,
			\item The natural map $\Sha(E/K_n) \rightarrow \Sha(E/K_{n + 1})$ is an injection for all $n \geq 0$.
		\end{enumerate}
		Then $X(E/K_\ac)$ does not have any non-trivial finite $\Lambda(\Gamma_\ac)$-submodule.
	\end{thm}
	
	\begin{rmk}
		In Section \ref{sec:def}, we shall use Proposition \ref{prop:Ven} and Proposition \ref{lem:H1=0} to prove the non-existence of non-trivial pseudo-null submodules of $X(E/K_\infty)$. In Section \ref{sec:indef}, we shall use Proposition \ref{lem:H1=0} to prove the non-existence of non-trivial finite submodules of $X(E/K_\ac)$. Our arguments vastly simplify the proofs of such results in the existing literature.
	\end{rmk}
	
	\subsection{Exceptional case} We recall that in this case, $E/\Q$ is an elliptic curve with CM by $\calO_K$. We also note that in this case, $K$ must be an imaginary quadratic field of class number $1$ by CM theory. Moreover, $K_\infty \subset K(E_{p^\infty})$ and there is an isomorphism $\Gal(K(E_{p^\infty})/K) \simeq \Gal(K_\infty/K) \times \Gal(K(E_p)/K)$.  It is a result due to Coates \cite{Coa83} that the dual Selmer group  $X(E/K_\infty)$ is always $\Lambda(G_\infty)$-torsion. When the sign of the functional equation of $L(E/\Q, s)$ is $+1$, the module $X(E/K_\ac)$ is $\Lambda(\Gamma_\ac)$-torsion \cite[Theorem 3]{Gre83}. Otherwise, when the sign of the functional equation of $L(E/\Q, s)$ is $-1$, the module $X(E/K_\ac)$ has $\Lambda(\Gamma_\ac)$-rank $2$ \cite[Theorem A]{AH06}. In fact, it was known to Greenberg that $X(E/K_\ac)$ is not $\Lambda(\Gamma_\ac)$-torsion when the sign is $-1$ \cite[Theorem 2]{Gre83}.
	
	Now, let $M_\infty$ be the maximal pro-$p$ abelian extension over $K_\infty$ that is unramified outside $p$. The Galois group $\Gal(M_\infty/K_\infty)$ has a natural $\Lambda(G_\infty)$-module structure and was studied by Greenberg \cite{Gre78}. The module structures of $X(E/K_\infty)$ and $\Gal(M_\infty/K_\infty)$ are related by the following result of Perrin-Riou:
	
	\begin{thm}\cite[Th\'eor\`eme 23]{P-R84}
		Suppose that the module $X(E/K_\infty)$ is $\Lambda(G_\infty)$-torsion. Then $X(E/K_\infty)$ does not have any non-trivial finite $\Lambda(G_\infty)$-submodules if and only if $\Gal(M_\infty/K_\infty)$ does not have any non-trivial finite $\Lambda(G_\infty)$-submodules.
	\end{thm}
	
	Greenberg showed that $\Gal(M_\infty/K_\infty)$ contains no non-trivial pseudo-null submodules \cite{Gre78}, from which it follows that
	
	\begin{thm} \label{thm:CM}
		The $\Lambda(G_\infty)$-module $X(E/K_\infty)$ does not have any non-trivial pseudo-null $\Lambda(G_\infty)$-submodules.
	\end{thm}
	
	For each $\alpha \in \calO_K$, let $E_\alpha$ denote the kernel of $\alpha$ as a $K$-endomorphism of $E$. For an ideal $\fraka \subset \calO_K$, denote by $E_{\fraka}$ the $\fraka$-torsion points of $E$, i.e. $E_\fraka = \cap_{\alpha \in \fraka} E_\alpha.$ Now, denote by $E_{\frakp^\infty} = \cup_{n \geq 0} E_{\frakp^n}$. 
	
	Assuming the Leopoldt conjecture holds for $K(E_{\frakp})$ and the weak Leopoldt conjecture holds for $K(E_{\frakp^\infty})/K(E_{\frakp})$ \cite[p. 44]{P-R84}, we have the following result.
	
	\begin{thm}\cite[Th\'eor\`eme 25]{P-R84} 
		The $\Lambda(\Gamma_\ac)$-module $X(E/K_\ac)$ does not have any non-trivial finite $\Lambda(\Gamma_\ac)$-submodules.
	\end{thm}
	\begin{proof} Under the condition that the Leopoldt conjecture holds for $K(E_{\frakp})$ and the weak Leopoldt conjecture holds for $K(E_{\frakp^\infty})/K(E_{\frakp})$, Perrin-Riou showed $\pd_{\Lambda(G_\infty)}(X(E/K_\infty)) \leq 1$. 
		
		Moreover, there is an isomorphism $X(E/K_\ac) \simeq X(E/K_\infty)_{H_\ac}$ by \cite[Lemme 9, Proposition 12]{P-R84}. Therefore,
		\[\pd_{\Lambda(\Gamma_\ac)}(X(E/K_\ac)) = \pd_{\Lambda(\Gamma_\ac)}(X(E/K_\infty)_{H_\ac}) = \pd_{\Lambda(G_\infty)}(X(E/K_\infty)) \leq 1\]
		by Lemme 5 in \cite{P-R84}.
		The conclusion now follows from another application of Proposition \ref{prop:Ven}.  
	\end{proof}

	\section{Galois cohomology of the Selmer group}
	\subsection{The definite setting} \label{sec:def}		
	Recall that in this setting, $N^{-}$ is a square-free product of an odd number of primes \cite{BD05}. We work under the extra hypothesis that $\Sel_p(E/K)$ is finite.
	\begin{hypo} \label{hypo:def}
		Assume that the following hold:
		\begin{itemize}
			\item[\mylabel{surj}{(surj)}] The residual Galois representation $\overline{\rho}_{E, p}: G_\Q \rightarrow Aut(E_p)$ is surjective.
			\item[\mylabel{finite}{(finite)}] The Selmer group $\Sel_p(E/K)$ is finite.
		\end{itemize}
	\end{hypo}
	
	Hypothesis \ref{finite} implies that $X(E/K_\ac)$ is $\Lambda(\Gamma_\ac)$-torsion, because the restriction map $\Sel_p(E/K) \rightarrow \Sel_p(E/K_\ac)^{\Gamma_\ac}$ is a pseudo-isomorphism. 
	
	Under some mild hypotheses, we show that $\Sel_p(E/K_\ac) \rightarrow \Sel_p(E/K_\infty)^{H_\ac}$ is an isomorphism and $H^1(H_\ac, \Sel_p(E/K_\infty)) = 0$.
	
	\begin{thm} \label{thm:coinv}
		Assume $p \nmid \prod_{v \mid p} \# \widetilde{E}_v(k_v)$.
		Then \begin{enumerate}
			\item The restriction map \[\Sel_p(E/K_\ac) \rightarrow \Sel_p(E/K_\infty)^{H_\ac}\] is an isomorphism,
			\item $H^1(H_\ac, \Sel_p(E/K_\infty)) = 0$.
		\end{enumerate}
	\end{thm}
	
	\begin{proof}
		Consider the following commutative diagram:
		\begin{center}
			\begin{tikzcd}[cramped, column sep=small]
				{0} \arrow[r] & {\Sel_p(E/K_\infty)^{H_\ac}} \arrow[r] & {H^1(K_S/K_\infty, E_{p^\infty})^{H_\ac}} \arrow[r, "\lambda_\infty^{H_\ac}"]           & {\bigoplus_{v \in S} J_v^1(E/K_\infty)^{H_\ac}}\\ 
				
				{0} \arrow[r] & {\Sel_p(E/K_\ac)} \arrow[r] \arrow[u, "\alpha"] & {H^1(K_S/K_\ac, E_{p^\infty})} \arrow[r, "\lambda_\ac"] \arrow[u, "\beta"] & \bigoplus_{v \in S} J_v^1(E/K_\ac) \arrow[u, "\gamma"] \arrow[r] & 0. 
			\end{tikzcd}
		\end{center}
		
		The map $\lambda_\ac$ in this diagram is surjective because $\Sel_p(E/K_\ac)$ is $\Lambda(\Gamma_\ac)$-cotorsion \cite[Theorem 7.2]{HV03}. We show that $\beta$ is an isomorphism and $\gamma: {\bigoplus_{v \in S} J^1_v(E/K_\ac)} \rightarrow \bigoplus_{v \in S} J_v^1(E/K_\infty)^{H_\ac}$ is injective. It will then follow from the snake lemma that $\alpha$ is an isomorphism. 
		
		First, $\ker(\beta) = H^1(K_\infty/K_\ac, E_{p^\infty}(K_\infty))$ by the Hochschild-Serre spectral sequence. It follows from \ref{surj} in Hypothesis \ref{hypo:def} that $E_{p^\infty}(K_\infty) = 0$ and thus $H^1(H_\ac, E_{p^\infty}(K_\infty)) = 0$. 
		
		Now, let us compute $\ker(\gamma) = \bigoplus_{v \in S} \gamma_v$. If $v \nmid p$, then for any primes $\eta$ in $K_\ac$ and $w$ in $K_\infty$ lying above $p$, the restriction map
		\[H^1(K_{\ac, \eta}, E_{p^\infty}) \rightarrow H^1(K_{\infty, w}, E_{p^\infty})\]
		is the identity map because both $K_{\ac, \eta}$ and $K_{\infty, w}$ are the unique pro-$p$ unramified extension of $K_v$.
		We consider the case $v \mid p$. Note that every prime $\eta \mid v$ in $K_\ac$  is ramified over $K$. Using Proposition 4.8 of \cite{CG96}, we have
		\[\frac{H^1(K_{\infty, w}, E_{p^\infty})}{\Ima(\kappa_{\infty, w})} = H^1(K_{\infty, w}, (\widetilde{E}_v)_{p^\infty}),\] \[\frac{H^1(K_{\ac, \eta}, E_{p^\infty})}{\Ima(\kappa_{\ac, \eta})} = H^1(K_{\ac, \eta}, (\widetilde{E}_v)_{p^\infty})\]
		for every prime $\eta$ in $K_\ac$ above $p$ and every prime $w$ in $K_\infty$ above $\eta$. For the second identity, we used the assumption that $v$ is ramified in $K_\ac$. The restriction map is simply 
		\[\gamma_w: H^1(K_{\ac, \eta}, (\widetilde{E}_w)_{p^\infty}) \rightarrow H^1(K_{\infty, w}, (\widetilde{E}_w)_{p^\infty}),\]
		whose kernel is $H^1(K_{\infty, w}/K_{\ac, \eta}, (\widetilde{E}_v)_{p^\infty}(k_{\infty, w}))$ where $k_{\infty, w}$ is the residual field of $K_{\infty, w}$. Since $(\widetilde{E}_v)_{p^\infty}(k_v) = 0$, we have \[H^1(K_{\infty, w}/K_{\ac, \eta}, (\widetilde{E}_v)_{p^\infty}(k_{\infty, w})) = 0\]  given our assumption and that $k_{\infty, w}/k_v$ is a cyclic pro-$p$ extension.

		In fact, $\gamma$ is an isomorphism because \[\coker(\gamma_v) = H^2(K_{\infty, w}/ K_{\ac, \eta}, (\widetilde{E}_v)_{p^\infty}(k_{\infty, w})) = 0.\] As a result, $\lambda_\infty^{H_\ac}$ is surjective because $\lambda_\ac$ is surjective and both $\beta, \gamma$ are isomorphisms.
		
		Now, we show that $H^1(H_\ac, \Sel_p(E/K_\infty)) = 0.$ 
		%Note that $\Sel_p(E/K_\cyc)$ is $\Lambda(\Gamma_\ac)$-cotorsion by Kato's work \cite{Kat04} and the fact that $K/\Q$ is abelian. 
		Using the same argument as Lemma 2.6 of \cite{HO10}, the fact that $\Sel_p(E/K_\ac)$ is $\Lambda(\Gamma_\ac)$-cotorsion implies $\Sel_p(E/K_\infty)$ is $\Lambda(G_\infty)$-cotorsion.
		
		By \cite[Theorem 7.2]{HV03}, we have $H^2(K_S/K_\infty, E_{p^\infty}) = 0$ and exactness of the sequence 
		
		\begin{equation} \label{seq:K_infty}
			{0} \rightarrow {\Sel_p(E/K_\infty)} \rightarrow {H^1(K_S/K_\infty, E_{p^\infty})} \rightarrow {\bigoplus_{v \in S} J_v^1(E/K_\infty)} \rightarrow 0.
		\end{equation} 
		
		By the same argument as \cite[Lemma 2.4]{CSS03}, $H^2(K_S/K_\infty, E_{p^\infty}) = 0$ implies \[H^1(H_\ac, H^1(K_S/K_\infty, E_{p^\infty})) = 0.\] Taking $H_\ac$-cohomology of short exact sequence \eqref{seq:K_infty} gives the identity \[\coker(\lambda_\infty^{H_\ac}) = H^1(H_\ac, \Sel_p(E/K_\infty)).\]
		This group is trivial because it has been shown that $\lambda_\infty^{H_\ac}$ is surjective.
	\end{proof}

	\begin{thm}
		Assume the same hypotheses as Theorem \ref{thm:coinv}. Then the module $X(E/K_\infty)$ does not have any non-trivial pseudo-null $\Lambda(G_\infty)$-submodules.
	\end{thm}
	
	\begin{proof}
		We  will show $\pd_{\Lambda(G_\infty)}(X(E/K_\infty)) \leq 1$, which implies that $X(E/K_\infty)$ does not have any non-trivial pseudo-null $\Lambda(G_\infty)$-submodules by Proposition \ref{prop:Ven}. By the Auslander-Buchsbaum formula \cite[Theorem 19.9]{Eis95}, the condition $\pd_{\Lambda(G_\infty)}(X(E/K_\infty)) \leq 1$ is equivalent to $\depth(X(E/K_\ac)) \geq 2.$ We will compute $\depth(X(E/K_\ac))$ using regular sequences.
		
		Let $\gamma_\cyc$ and $\gamma_\ac$ be topological generators of $\Gamma_\cyc$ and  $\Gamma_\ac$, respectively. One may write $G_\infty = \Gamma_\ac \times H_\ac$ where $H_\ac \simeq \Gamma_\cyc$, so that $\gamma_\cyc$ can be regarded as a topological generator of $H_\ac$.  Denote $T_\cyc := \gamma_\cyc - 1$ and $T_\ac := \gamma_\ac - 1$. We show that $\{T_\cyc, T_\ac\}$ is an $X(E/K_\infty)$-regular sequence in the maximal ideal $\frakm$ of $\Lambda(G_\infty)$
		
		By Theorem \ref{thm:tor}, $X(E/K_\infty)[T_\cyc] = H_1(H_\ac, X(E/K_\infty)) = 0$.
		Therefore, $T_\cyc$ is a non-zero divisor on $X(E/K_\infty)$. Moreover, $X(E/K_\infty)/T_\cyc X(E/K_\infty)$ is non-zero because otherwise $X(E/K_\infty)$ would be trivial by Nakayama's lemma. 
		
		On the other hand, the isomorphism $X(E/K_\infty)/T_\cyc X(E/K_\infty) \simeq X(E/K_\ac)$ (Theorem \ref{thm:coinv}) implies $T_\ac$ is a non-zero divisor on $X(E/K_\infty)/T_\cyc X(E/K_\infty)$, as we have seen that $X(E/K_\ac)[T_\ac] = H_1(\Gamma_\ac, X(E/K_\ac)) = 0$ (Theorem \ref{thm:tor}). Once again, Nakayama's lemma implies \[X(E/K_\infty)/(T_\cyc, T_\ac) X(E/K_\infty)\] 
		is non-trivial. 
	\end{proof}
	
	\subsection{The indefinite setting} \label{sec:indef}
	
	Throughout this section, we consider the hypotheses below, which are part of Assumptions A and B in the work of Bertolini \cite{Ber01}:
	
	\begin{hypo} \label{hypo:indef}
		Assume that the following hold:
		\begin{itemize}[align=left]
			\item[\mylabel{good}{(good)}] $p \nmid 6 N \disc(K) \#\Pic(\calO_K).$
			\item[\mylabel{ord}{(ord)}] $E$ has good ordinary reduction at $p$.
			\item[\mylabel{Heeg}{(Heeg)}] The square-free part of $N^-$ is a product of an even number of primes.
		\end{itemize}	
	\end{hypo}
	%In \cite{Ber01}, Bertolini worked with the weak Heegner hypothesis, which only requires $N^{-}$ to be the product of an even number primes. For simplicity, we will work with the strong Heegner hypothesis \ref{Heeg}. 
	Hypothesis \ref{Heeg} is called the weak Heegner hypothesis. In this setting, there is an explicitly constructed Heegner point $y_K \in E(K)$. In the simplest case when $N^{-} = 1$ (the strong Heegner hypothesis), there is an ideal $\frakN \subset \calO_K$ such that $N \calO_K = \frakN \overline{\frakN}$ and $y_K$ is the image of $[(\C/\calO_K, \frakN^{-1}/\calO_K)]$ under the modular parametrization $\varphi : X_0(N) \rightarrow E$ sending $i\infty$ to the origin of $E$. We will assume that
	\begin{itemize}
		\item[\mylabel{Heeg-pt}{(inf)}] $y_K$ has infinite order in $E(K)$.
	\end{itemize}
	
	Note that assumption \ref{Heeg-pt} makes the hypotheses of our results numerically verifiable, in contrast with the hypotheses for similar results in the literature. See Theorem \ref{cor:main} and the following examples (Example \ref{ex:indef}, Example \ref{ex:indef-2}). 
	
	\begin{rmk} \label{rmk:Kol}
		The works of Kolyvagin \cite{Kol90}, Kolyvagin-Logachev \cite{KL89} and \\ Howard \cite{How04b} imply the following:
		\begin{itemize}[align=left]
			\item[\mylabel{M-W}{(M-W)}] $E/K$ has Mordell-Weil rank equal to $1$,
			\item[\mylabel{Sha}{(Sha)}] $\Sha(E/K)_{p^\infty}$ is finite.
		\end{itemize}
	\end{rmk}

	Under the hypothesis $N^{-} = 1$, Perrin-Riou  constructed the free $\Lambda(\Gamma_\ac)$-module of Heegner points $\mathcal{H}_\ac \subset \calS_p(E/K_\ac)$ \cite{P-R87}. This construction was later generalized by Howard under the weak Heegner hypothesis \cite{How04b}. It follows from the norm relations on Heegner points that $\calH_\ac$ is $\Lambda(\Gamma_\ac)$-free of rank $1$ with our assumption on $y_K$ (see for example, \cite[Theorem 3.4.3]{How04b}). Even without this assumption, Cornut and Vatsal \cite{Cor02, Vat03} showed that $\mathcal{H}_\ac$ is always non-trivial. Moreover, Howard showed that both $\widehat{\Sel_p(E/K_\ac)}$ and $\calS_p(E/K_\ac)$ have $\Lambda(\Gamma_\ac)$-rank $1$ assuming \ref{Heeg-pt} \cite{How04, How04b}. The latter is a free $\Lambda(\Gamma_\ac)$-module because there is an injection $\calS_p(E/K_\ac) \hookrightarrow \Hom_{\Lambda}(X(E/K_\ac), \Lambda)$ \cite[Lemme 5]{P-R87}.	
	
	The non-triviality of Heegner points also implies the weak Leopoldt conjecture, i.e. the vanishing of $H^2(K_S/K_\ac, E_{p^\infty})$ under certain hypotheses \cite[Theorem 5.4]{Ber01}. We begin with a simpler proof of this fact assuming \ref{Heeg-pt}.
	
	\begin{lemma} \label{lem:trivial} Under our assumptions as in Hypothesis \ref{hypo:indef}, we have the following:
		\begin{enumerate}
			\item $H^2(K_S/K, E_{p^\infty}) = 0$.
			%\item $H^2(K_S/K_\ac, E_{p^\infty}) = 0$, i.e. the weak Leopoldt conjecture holds for $E_{p^\infty}$ over $K_\ac$.
			\item $H^1(\Gamma_\ac, H^1(K_S/K_\ac, E_{p^\infty})) = 0$.
			\item $H^2(K_S/K_\ac, E_{p^\infty}) = 0$.
		\end{enumerate}
	\end{lemma}
	
	\begin{proof}
		Recall that  $E(K)$ has $\Z$-rank $1$ and $\Sha(E/K)(p)$ is finite (see Remark \ref{rmk:Kol}). The same proof as Theorem 12 of \cite{CM94} allows us to conclude that $H^2(K_S/K, E_{p^\infty}) = 0$, which is the strong Leopoldt conjecture. 
		
		%		For part (ii), consider the exact sequence \ref{eqn:Sel}.
		%		
		%		For any prime $v \nmid p$, it follows from \cite[Proposition 2]{Gre89} that $\corank_{\Lambda(\Gamma_\ac)} J_v^1(E/K_\ac) = 0$. On the other hand, for a prime $v \mid p$, we have $\corank_{\Lambda(\Gamma_\ac)} J_v^1(E/K_\ac) = [K_v: \Q_p]$ by \cite[Theorem 4.9]{CG96}. In total, we have $\corank_{\Lambda(\Gamma_\ac)} \bigoplus_{v \in S} J_v^1(E/K_\ac) = \sum_{v \mid p} [K_v: \Q_p] = [K:\Q] = 2.$ 
		%		
		%		It is known that both $\widehat{\Sel_p(E/K_\ac)}$ and $\calS_p(E/K_\ac)$ are free of rank $1$ over $\Lambda(\Gamma_\ac)$ \cite{How04}. The exact sequence \ref{eqn:Sel} thus implies that $\corank_{\Lambda(\Gamma_\ac)} H^1(K_S/K_\ac, E_{p^\infty}) = 2$. Proposition 3 of \cite{Gre89} gives the inequality
		%		\[\corank_{\Lambda(\Gamma_\ac)} H^1(K_S/K_\ac, E_{p^\infty}) - \corank_{\Lambda(\Gamma_\ac)} H^2(K_S/K_\ac, E_{p^\infty}) \geq 2,
		%		\]
		%		which means $\corank_{\Lambda(\Gamma_\ac)} H^2(K_S/K_\ac, E_{p^\infty}) = 0$. The fact that $H^2(K_S/K_\ac, E_{p^\infty})$ is also $\Lambda(\Gamma_\ac)$-cofree \cite[Proposition 4]{Gre89} allows us to conclude that $H^2(K_S/K_\ac, E_{p^\infty}) = 0$. 
		
		We look at the Hochschild-Serre spectral sequence
		\[H^i(\Gamma_\ac, H^j(K_S/K_\ac, E_{p^\infty})) \Rightarrow H^{i + j}(K_S/K, E_{p^\infty}).\]		
		The fact that $H^2(K_S/K, E_{p^\infty}) = 0$ and $H^2(\Gamma_\ac, H^0(K_S/K_\ac, E_{p^\infty})) = 0$ ($\Gamma_\ac$ has $p$-cohomological dimension $1$) implies both part (ii) and part (iii) of the Lemma. 
	\end{proof}

	Throughout the rest of this section, we will  analyze the fundamental diagram
	\begin{equation} \label{diag:fundamental}
		\begin{tikzcd}[column sep = small, cramped]
			{0} \arrow[r] & {\Sel_p(E/K_\ac)^{\Gamma_\ac}} \arrow[r] & {H^1(K_S/K_\ac, E_{p^\infty})^{\Gamma_\ac}} \arrow[r, "\lambda_\ac^{\Gamma_\ac}"]           & {\bigoplus_{v \in S} J_v^1(E/K_\ac)^{\Gamma_\ac}} \\
			{0} \arrow[r] & {\Sel_p(E/K)} \arrow[r] \arrow[u, "\alpha"] & {H^1(K_S/K, E_{p^\infty})} \arrow[r, "\lambda"] \arrow[u, "\beta"] & {\bigoplus_{v \in S} H^1(K_v, E)(p)} \arrow[u, "\gamma"].
		\end{tikzcd}
	\end{equation}
	
	One may split \eqref{diag:fundamental} as follows:
	\begin{equation} \label{diag:fundamental2}
		\begin{tikzcd}[column sep = small, cramped]
			{0} \arrow[r] & {\Sel_p(E/K_\ac)^{\Gamma_\ac}} \arrow[r] & {H^1(K_S/K_\ac, E_{p^\infty})^{\Gamma_\ac}} \arrow[r, "\phi_\ac"]           & {\Ima(\lambda_\ac^{\Gamma_\ac})} \arrow[r] & 0\\
			{0} \arrow[r] & {\Sel_p(E/K)} \arrow[r] \arrow[u, "\alpha"] & {H^1(K_S/K, E_{p^\infty})} \arrow[r, "\lambda"] \arrow[u, "\beta"] & {\Ima(\lambda)} \arrow[u, "\delta"] \arrow[r] & 0,
		\end{tikzcd}
	\end{equation}
	
	\begin{equation} \label{diag:fundamental3}
		\begin{tikzcd}[cramped]
			{0} \arrow[r] & {\Ima(\lambda_\ac^{\Gamma_\ac})} \arrow[r] & {\bigoplus_{v \in S} J_v^1(E/K_\ac)^{\Gamma_\ac}} \arrow[r]           & {\coker(\lambda_\ac^{\Gamma_\ac})} \arrow[r] & 0\\
			{0} \arrow[r] & {\Ima(\lambda)} \arrow[r] \arrow[u, "\delta"] & {{\bigoplus_{v \in S} H^1(K_v, E)}(p)} \arrow[r] \arrow[u, "\gamma"] & {\coker(\lambda)} \arrow[u, "\epsilon"] \arrow[r] & 0.
		\end{tikzcd}
	\end{equation}
	Let $\phi_\ac$ denote the map ${H^1(K_S/K_\ac, E_{p^\infty})^{\Gamma_\ac}} \rightarrow {\Ima(\lambda_\ac^{\Gamma_\ac})}$ in the top row of \eqref{diag:fundamental2}, thereby distinguishing it from the map $\lambda_\ac^{\Gamma_\ac}$ in \eqref{diag:fundamental}. For each prime $v$ in $K$ dividing $N$, let $c_v(E/K)$ be the local Tamagawa factor of $E/K$ at $v$. Moreover, for each prime $v\mid p$, we let $d_v = \# \widetilde{E}_v(k_v)$, where $\widetilde{E}_v$ is the reduction of $E/K$ at $v$ and $k_v$ is the residual field of $K$ with respect to $v$.

	\begin{thm} \label{thm:gamma}
		The restriction map \[\gamma: \bigoplus_{v \in S} H^1(K_v, E) \rightarrow \bigoplus_{v \in S} J_v^1(E/K_\ac)^{\Gamma_\ac}\] is surjective with finite kernel of order $\prod_{v \mid N^{+}} c_v(E/K)^{(p)} \times \prod_{v \mid p} (d_v^{(p)})^2$.
	\end{thm}
	
	\begin{proof}
		For each prime $v$ in $K$, denote by \[\gamma_v: H^1(K_v, E) \rightarrow \bigoplus_{v \in S} J_v^1(E/K_\ac)^{\Gamma_\ac}\] the local component of $\gamma$ at $v$.

		We first consider a prime $v$ not lying above $p$. If the rational prime below $v$ is inert in $K$, then $v$ is infinitely split in $K_\ac/K$ and $\gamma_v$ is simply the identity map. Otherwise, $v$ is finitely decomposed in $K_\ac/K$ and the same argument as \cite[Lemma 3.4]{CS00} shows that $\gamma_v$ is surjective and $\ker(\gamma_v)$ is finite of order $c_v(E/K)^{(p)}$, where $c_v(E/K)^{(p)}$ is the largest $p$-power that divides $c_v(E/K)$. Note that $c_v(E/K)$ is trivial for primes $v \nmid N$.
		
		Next, consider a prime $v$ lying above $p$. Since every prime above $p$ is ramified in $K_\ac/K$, the same argument as \cite[Proposition 3.5]{CS00} shows that $\gamma_v$ is surjective and $\ker(\gamma_v)$ is finite of order $(d_v^{(p)})^2$. The conclusion now follows.
	\end{proof}
	
	\begin{cor} \label{cor:coinv}
		Both $\ker(\alpha)$ and $\coker(\alpha)$ are finite, and $\# \coker(\alpha)$ divides $\prod_{v \mid N^{+}} c_v(E/K)^{(p)} \times \prod_{v \mid p} (d_v^{(p)})^2$.
	\end{cor}
	
	\begin{proof}
		The proof follows directly from Theorem \ref{thm:gamma} by applying the snake lemma to \eqref{diag:fundamental} once we show that both $\ker(\beta)$ and $\coker(\beta)$ are finite. 
		
		By the Hochschild-Serre spectral sequence, $\ker(\beta) = H^1(\Gamma_\ac, E_{p^\infty}(K_\ac))$. This group is always finite, following the argument in \cite[Lemma 3.1]{Gre99}. Note that the argument does not assume the finiteness of $E_{p^\infty}(K_\ac)$. Moreover, $\coker(\beta) =  H^2(\Gamma_\ac, E_{p^\infty}(K_\ac)) = 0$ since $\Gamma_\ac \simeq \Z_p$ has $p$-cohomological dimension $1$.
	\end{proof}
	
	\begin{lemma} \label{lem:E(K)(p)=0} Suppose that $E_{p^\infty}(K) = 0$. Then:
		\begin{enumerate}
			\item The map $\beta$ is an isomorphism
			\item The map $\alpha$ is injective.
			\item There is an isomorphism $\ker(\delta) \simeq \coker(\alpha).$ 
		\end{enumerate}
	\end{lemma}
	
	\begin{proof}
		Part (i) directly follows from the fact that \[\ker(\beta) = H^1(\Gamma_\ac, E_{p^\infty}(K_\ac)) = 0\] under the assumption $E_{p^\infty}(K) = 0$. Part (ii) and  (iii) follow from part (i) by chasing \eqref{eqn:Sel-1}.
	\end{proof}
	
	The following is a natural corollary of Corollary \ref{cor:coinv} and Lemma \ref{lem:E(K)(p)=0}.
	\begin{cor} \label{cor:isom}
		When $E_{p^\infty}(K) = 0$ and $\prod_{v \mid N^{+}} c_v(E/K)^{(p)} \times \prod_{v \mid p} (d_v^{(p)})^2 = 1,$ the control morphism $\alpha: \Sel_p(E/K) \rightarrow \Sel_p(E/K_\ac)^{\Gamma_\ac}$ is an isomorphism.
	\end{cor}
	
	\begin{lemma} \label{lem:epsilon}
		The kernel of $\epsilon$ is finite of order $\#\ker(\gamma)/\#\ker(\delta)$.
	\end{lemma}
	
	\begin{proof}
		We saw that $\ker(\gamma)$ is finite (Theorem \ref{thm:gamma}), and $\beta$ is always surjective (Corollary \ref{cor:coinv}). Hence, it follows from \eqref{diag:fundamental2} that $\delta$ is surjective. Moreover, \eqref{diag:fundamental3} gives the following exact sequence:
		\[0 \rightarrow \ker(\delta) \rightarrow \ker(\gamma) \rightarrow \ker(\epsilon) \rightarrow 0,\]
		and the Lemma immediately follows.
		%Finally, there is an isomorphism $\ker(\delta) \simeq \coker(\alpha)$ by applying the snake lemma to \eqref{diag:fundamental2}, together with the fact that $\beta$ is an isomorphism. The Lemma then follows immediately.
	\end{proof}
	
	However, we can say more about the size of $\ker(\delta)$. Every prime $\ell \mid N^{+}$ splits in $K/\Q$ as $\ell \calO_K = v \vbar$, and so $c_v(E/K) = c_{\vbar}(E/K) = c_\ell(E/\Q)$ where $c_\ell(E/\Q)$ is the Tamagawa factor at $\ell$ of $E/\Q$. Thus $\# \ker(\gamma) = (\prod_{\ell \mid N^{+}} c_\ell(E/\Q)^{(p)} \times \prod_{v \mid p} d_v^{(p)})^2$ is a square, and the following result shows that $\ker(\delta)$ is bounded by its squareroot. 
	
	\begin{lemma} \label{lem:diag}
		If $p$ is split in $\calO_K$, then \[\# \ker(\delta) \mid \prod_{\ell \mid N^{+}} {c_\ell(E/\Q)^{(p)}} \times \prod_{v \mid p} d_v^{(p)}.\] 
		If $p$ is inert in $\calO_K$ and $d_v^{(p)} = 1$ for the unique prime $v \mid p$, then \[\# \ker(\delta) \mid \prod_{\ell \mid N^{+}} {c_\ell(E/\Q)^{(p)}}.\]
	\end{lemma} 
	\begin{proof}
		Note that $\ker(\delta) = \Ima(\lambda) \cap \ker(\gamma)$. For each prime $\ell \mid N^{+}$ that splits as $(\ell) = v \vbar$ in $\calO_K$, we have $G_v = G_{\vbar}$ in $\Gal(K_S/K)$ where $G_v$ and $G_{\vbar}$ are respectively the decomposition groups of $v$ and $\vbar$. Hence, $\Ima(\lambda) \cap (\ker(\gamma_v) \times \ker(\gamma_{\vbar}))$ lies in the diagonal of $H^1(\Gamma_v, E)(p) \times H^1(\Gamma_{\vbar}, E)(p)$. When $p$ is split in $K$, a similar argument also holds for the primes above $p$, and the conclusion follows.
		
		On the other hand, it is not clear if the conclusion holds generally when $p$ is inert. However, under the simplifying assumption $d_v^{p} = 1$ for $v \mid p$, the same argument implies $\# \ker(\delta) \mid \prod_{\ell \mid N^{+}} {c_\ell(E/\Q)^{(p)}}.$
	\end{proof}
	
	%In Lemma \ref{lem:gamma} and Lemma \ref{lem:coinv} below, we will see that both $\ker(\gamma)$ and $\ker(\alpha)$ are finite. The main goal of this section is to prove the following:

	%	\begin{thm} \label{thm:coker}
		%		Assume hypotheses \ref{good}, \ref{ord}, \ref{Heeg} and \ref{surj}. Then $\widehat{\coker(\lambda_\ac^{\Gamma_\ac})} \simeq \Z_p.$
		%	\end{thm}
	%	\begin{proof}
		%		Once again, we analyze the fundamental diagram (\ref{diag:fundamental}). By the same argument as the proof of Lemma \ref{lem:coinv}, $\beta$ is an isomorphism. Moreover, Lemma \ref{lem:gamma} implies that $\gamma$ is surjective with finite kernel. 
		%		Since $\lambda_\ac^{\Gamma_\ac} \circ \beta = \gamma \circ \lambda$, the following sequence is exact:
		%		\begin{equation} \label{eqn:comp}
			%			\ker(\gamma) \rightarrow \coker(\lambda) \rightarrow \coker(\lambda_\ac^{\Gamma_\ac}) \rightarrow \coker(\gamma)		
			%		\end{equation} 
		%		 Hypothesis \ref{Heeg} implies that $\Sha(E/K)(p)$ is finite (see Remark \ref{rmk:Kol}). It follows from the proof of Proposition 1.9 in \cite{CS00} that $\coker(\lambda) = \widehat{E(K)^\ast}$, where $E(K)^\ast = \varprojlim_{n} E(K)/p^n E(K)$. Moreover, hypothesis \ref{Heeg} implies that $E/K$ must have Mordell-Weil rank $1$ and \ref{surj} implies $E_{p^\infty}(K) = 0$. It follows that $E(K)^{\ast} \simeq \Z_p$ and therefore $\widehat{\coker(\lambda)} \simeq \Z_p.$ Moreover, $\ker(\gamma)$ is finite and $\coker(\gamma) = 0$ in the short exact sequence (\ref{eqn:comp}), which allows us to conclude that $\widehat{\coker(\lambda_\ac^{\Gamma_\ac})} \simeq \Z_p$.
		%	\end{proof}

	Under the weak Leopoldt conjecture (proved in Lemma \ref{lem:trivial}), there is an exact sequence
	\begin{equation} \label{eqn:Sel}
		\begin{aligned}
			0 \rightarrow \Sel_p(E/K_\ac) \rightarrow & H^1(K_S/K_\ac, E_{p^\infty})  \\ & \xrightarrow{\lambda_\ac} \bigoplus_{v \in S} J_v^1(E/K_\ac) \rightarrow \widehat{\calS_p(E/K_\ac)} \rightarrow 0.
		\end{aligned}
	\end{equation}		
	
	%	Since we assumed \ref{Heeg-pt}, it is not too difficult to prove the following classical result. 
	%	\begin{thm}
		%		Both $X(E/K_\ac)$ and $\calS_p(E/K_\ac)$ have rank $1$ as $\Lambda(\Gamma_\ac)$-modules, and the latter is cofree.
		%	\end{thm}
	%	\begin{proof}
		%		We use the fact that $\calS_p(E/K_\ac) = \Hom_{\Lambda}(X(E/K_\ac), \Lambda)$ to conclude that $\calS_p(E/K_\ac)$ is cofree. Since $X(E/K_\ac)^{\Gamma_\ac}$ and $X(E/K)$ are pseudo-isomorphic (Corollary \ref{cor:coinv}) and the latter has $\Z_p$-rank $1$, the $\Lambda(\Gamma_\ac)$ rank of $X(E/K_\ac)$ is at most $1$. The non-triviality of the Heegner points $\mathcal{H}_\ac$ gives the reverse inequality, and the conclusion follows.
		%	\end{proof}
	
	We now state and prove the main theorem of this section.
	
	\begin{thm} \label{thm:main}
		The homology group $H_1(\Gamma_\ac, X(E/K_\ac))$ is finite of order \[\frac{\# \ker(\delta) \cdot[\calS_p(E/K): \Ima(\calS_p(E/K_\ac)_{\Gamma_\ac})]}{\#\ker(\gamma)}.\] Thus, $H_1(\Gamma_\ac, X(E/K_\ac)) = 0$ if and only if \[\# \ker(\gamma) = [\calS_p(E/K) : \Ima(\calS_p(E/K_\ac)_{\Gamma_\ac})] \# \ker(\delta).\]
	\end{thm}
	
	\begin{proof}
		Recall the exact sequence \ref{eqn:Sel}, which we may split into two short exact sequences
		\begin{equation} \label{eqn:Sel-1}
			0 \rightarrow \Sel_p(E/K_\ac) \rightarrow H^1(K_S/K_\ac, E_{p^\infty}) \xrightarrow{\lambda_\ac} \Ima(\lambda_\ac) \rightarrow 0,
		\end{equation}
		\begin{equation} \label{eqn:Sel-2}
			0 \rightarrow \Ima(\lambda_\ac) \rightarrow \bigoplus_{v \in S} J_v^1(E/K_\ac) \rightarrow \widehat{\calS_p(E/K_\ac)} \rightarrow 0.
		\end{equation}

		Because $H^1(\Gamma_\ac, H^1(K_S/K_\ac, E_{p^\infty})) = 0$ (Lemma \ref{lem:trivial}), taking $\Gamma_\ac$-invariants of the short exact sequence \eqref{eqn:Sel-1} gives the following exact sequence:
		\begin{align} \label{eqn:inv-1}  
			0 \rightarrow \Sel_p(E/K_\ac)^{\Gamma_\ac} \rightarrow H^1(K_S/K_\ac, E_{p^\infty})^{\Gamma_\ac} & \xrightarrow{} (\Ima(\lambda_\ac))^{\Gamma_\ac} 
			\\ & \rightarrow H^1(\Gamma_\ac, \Sel_p(E/K_\ac)) \rightarrow 0. \notag
		\end{align}
		Hence, $H^1(\Gamma_\ac, \Sel_p(E/K_\ac)) \simeq (\Ima(\lambda_\ac))^{\Gamma_\ac}/\Ima(\phi_\ac).$ On the other hand, taking $\Gamma_\ac$-invariants of the short exact sequence \eqref{eqn:Sel-2} gives exactness of
		\begin{align} \label{eqn:inv-2}
			0 \rightarrow (\Ima(\lambda_\ac))^{\Gamma_\ac} \rightarrow \bigoplus_{v \in S} J_v^1(E/K_\ac)^{\Gamma_\ac} \rightarrow \widehat{\calS_p(E/K_\ac)}^{\Gamma_\ac} \rightarrow H^1(\Gamma_\ac, \Ima(\lambda_\ac)).
		\end{align}
		Now, \eqref{eqn:Sel-1} again implies that there is an exact sequence
		\begin{align}
			H^1(\Gamma_\ac, H^1(K_S/K_\ac, E_{p^\infty})) \rightarrow H^1(\Gamma_\ac, \Ima(\lambda_\ac)) \rightarrow H^2(\Gamma_\ac, \Sel_p(E/K_\ac)).
		\end{align}
		Lemma \ref{lem:trivial} implies $H^1(\Gamma_\ac, H^1(K_S/K_\ac, E_{p^\infty})) = 0$. Moreover, \[H^2(\Gamma_\ac, \Sel_p(E/K_\ac)) = 0\] because $\Gamma_\ac$ has $p$-cohomological dimension $1$. Thus, $H^1(\Gamma_\ac, \Ima(\lambda_\ac)) = 0$, which lets us rewrite
		\eqref{eqn:inv-2} as
		\begin{equation}
			0 \rightarrow (\Ima(\lambda_\ac))^{\Gamma_\ac} \rightarrow \bigoplus_{v \in S} J_v^1(E/K_\ac)^{\Gamma_\ac} \rightarrow \widehat{\calS_p(E/K_\ac)}^{\Gamma_\ac} \rightarrow 0.
		\end{equation}
		It follows that
		\begin{equation}
			0 \rightarrow \frac{(\Ima(\lambda_\ac))^{\Gamma_\ac}}{\Ima(\phi_\ac)} \rightarrow \frac{\bigoplus_{v \in S} J_v^1(E/K_\ac)^{\Gamma_\ac}}{\Ima(\phi_\ac)} \rightarrow \widehat{\calS_p(E/K_\ac)}^{\Gamma_\ac} \rightarrow 0
		\end{equation}
		is exact. The group $\bigoplus_{v \in S} J_v^1(E/K_\ac)^{\Gamma_\ac}/\Ima(\phi_\ac)$ is in fact $\coker(\lambda_\ac^{\Gamma_\ac})$ in (\ref{eqn:Sel}), and therefore $H^1(\Gamma_\ac, \Sel_p(E/K_\ac)) \simeq (\Ima(\lambda_\ac))^{\Gamma_\ac}/\Ima(\phi_\ac)$ is the kernel of the surjection
		\[ \coker(\lambda_\ac^{\Gamma_\ac}) \rightarrow \widehat{\calS_p(E/K_\ac)}^{\Gamma_\ac}\]
		induced by (\ref{eqn:Sel}). Consider the following commutative diagram:
		\begin{equation} \label{diag:coker}
			\begin{tikzcd}
				{\coker(\lambda_\ac^{\Gamma_\ac})} \arrow[r, two heads] & {\widehat{\calS_p(E/K_\ac)}^{\Gamma_\ac}} \\
				{\coker(\lambda) = \widehat{\calS_p(E/K)}} \arrow[ru, two heads] \arrow[u, two heads, "\epsilon"] &   
			\end{tikzcd}
		\end{equation}
		We saw in Lemma \ref{lem:epsilon} that $\epsilon$ is surjective with kernel of finite order \[\#\ker(\gamma)/\#\coker(\alpha).\] On the other hand, the kernel of the diagonal map is finite of order \[[\calS_p(E/K) : \Ima(\calS_p(E/K_\ac)_{\Gamma_\ac})].\]
		Hence, the kernel of the horizontal map, which is $H^1(\Gamma_\ac, \Sel_p(E/K_\ac))$, is also finite. Its order is 
		\[\frac{[\calS_p(E/K) : \Ima(\calS_p(E/K_\ac)_{\Gamma_\ac})]}{\#\ker(\epsilon)} = \frac{[\calS_p(E/K) : \Ima(\calS_p(E/K_\ac)_{\Gamma_\ac})]\# \ker(\delta)}{\#\ker(\gamma)},\]
		as required.
	\end{proof}
	
	Write $K_\ac = \cup_{n \geq 0} K_n$ and let $U \calS_p(E/K)$ be the module of universal norms 
	\[U \calS_p(E/K) := \bigcap_{n \geq 0} \Cores_{K_{n}/K} \calS_p(E/K_{n})\]
	inside $\calS_p(E/K)$. If $\Ima(\calS_p(E/K_{\ac})_{\Gamma_\ac}) = \calS_p(E/K)$, then the corestriction maps $\Cores_{K_n/K}$ are surjective, which implies $U\calS_p(E/K) = \calS_p(E/K).$  The converse also holds, for if $\Cores_{K_n/K}$ are surjective then the natural map
	\[\varprojlim_{n} \calS_p(E/K_n) \rightarrow \calS_p(E/K)\]
	is also surjective because each $\calS_p(E/K_n)$ is a finite-length $\Z_p$-module. With  respect to studying finite submodules of $X(E/K_\ac)$, Bertolini proved 
	\begin{thm} \cite[Corollary 6.2]{Ber01} \label{thm:Ber01}
		Assume
		\begin{enumerate} 
			%\item $N^{-}$ is the product of an even number of primes,
			\item $\overline{\rho}_{E, p}: G_\Q \rightarrow Aut(E_p)$ is surjective,
			\item $p \nmid \prod_{v \mid p} d_v \cdot \prod_{\ell \mid N} c_\ell(E/\Q).$
		\end{enumerate}
		Then the module $X(E/K_\ac)$ does not have any non-trivial finite $\Lambda(\Gamma_\ac)$-submodule if and only if $\calS_p(E/K)/U\calS_p(E/K)$ is torsion-free.
	\end{thm}
	Section 7 in {\it loc. cit.} also gives some conditions under which the module $\calS_p(E/K)/U\calS_p(E/K)$ is torsion-free. In particular,
	
	\begin{thm} \cite[Theorem 7.1]{Ber01}
		Assume the same hypotheses as Theorem \ref{thm:Ber01} and
		\begin{enumerate}
			\item The module of Heegner points over $K_n$ is non-zero modulo $pE(K_n)$ for some $n \geq 0$,
			\item The natural map $\Sha(E/K_n) \rightarrow \Sha(E/K_{n + 1})$ is an injection for all $n \geq 0$.
		\end{enumerate}
		Then $\calS_p(E/K)/U \calS_p(E/K)$ is torsion-free, and $X(E/K_\ac)$ does not have any non-trivial finite $\Lambda(\Gamma_\ac)$-submodule.
	\end{thm}

	In our setting, when \ref{Heeg-pt} holds,  both $\calS_p(E/K)$ and $US_p(E/K)$ have $\Z_p$-rank $1$. Therefore, the criterion $\calS_p(E/K)/U\calS_p(E/K)$ being torsion-free in Theorem \ref{thm:Ber01} is equivalent to $\calS_p(E/K) = U\calS_p(E/K)$. We will give another proof of Theorem \ref{thm:Ber-tor} in our setting using Theorem \ref{thm:main}. First, we recall that $\Sel_p(E/K_\ac)$ has $\Lambda(\Gamma_\ac)$-corank $1$ which is also the $\Z_p$-corank of $\Sel_p(E/K)$. Recall from Proposition \ref{lem:H1=0} that $X(E/K_\ac)$ does not have any finite $\Lambda(\Gamma_\ac)$-submodule if $H^1(\Gamma_\ac, \Sel_p(E/K_\ac)) = 0$. We show that the condition $H^1(\Gamma_\ac, \Sel_p(E/K_\ac)) = 0$ is equivalent to $\calS_p(E/K) = U\calS_p(E/K)$ under the same hypotheses as \cite{Ber01}. 
	
	\begin{thm} \label{thm:Ber} Suppose that $\overline{\rho}_{E, p}: G_\Q \rightarrow \Aut(E_{p^\infty})$ is surjective and $p \nmid \prod_{v \mid p} d_v \cdot \prod_{\ell \mid N} c_\ell(E/\Q).$ Then $X(E/K_\ac)$ does not have any non-trivial finite $\Lambda(\Gamma_\ac)$-submodule if and only if $U\calS_p(E/K) = \calS_p(E/K)$.
	\end{thm}
	\begin{proof} 
		
		Under our assumptions, we may use Corollary \ref{cor:coinv} to conclude that $\ker(\gamma)$ and $\ker(\delta)$ are trivial. The identity in Theorem \ref{thm:main} then becomes $H^1(\Gamma_\ac, \Sel_p(E/K_\ac)) = [\calS_p(E/K): \Ima(\calS_p(E/K_\ac)_{\Gamma_\ac})]$, and we have seen that this is trivial if and only if $U\calS_p(E/K) = \calS_p(E/K).$
	\end{proof}

	It is natural to ask if $X(E/K_\ac)$ has no non-trivial finite $\Lambda(\Gamma_\ac)$-submodules even if one relaxes the hypotheses in Theorem \ref{thm:Ber}. To this end, we impose the strong Heegner hypothesis $N^{-} = 1$ and recall that the Birch \& Swinnerton-Dyer (BSD) conjecture is equivalent to the identity
	\[[E(K) : \Z y_K ] = (\#\Sha(E/K))^{1/2} \#(\calO_K^\times /\Z^\times)  c_{\text{Manin}}(\varphi)  \prod_{\ell \mid N} c_\ell(E/\Q)\]
	where $c_{\text{Manin}}(\varphi) \in \Z_{> 0}$ is the Manin constant for $\varphi$. It is worth noting that $ \#(\calO_K^\times /\Z^\times) = 1$ under our assumption on $K$ and that $c_{\text{Manin}}(\varphi)$ is conjectured to be $1$. We refer the readers to \cite[(2.2)]{GZ86} for this formulation of the BSD conjecture and \cite{ARS06} for the conjecture regarding the Manin constant. %We consdier the $p$-primary version of the BSD conjecture:
	%\[[E(K) \otimes \Z_p: \Z_p (y_K \otimes 1)] = (\#\Sha(E/K)(p))^{1/2}  u_K c_{\text{Manin}}(\varphi)  \prod_{\ell \mid N} c_\ell(E/\Q)^{(p)},\]
	
	Granting these conjectures, $p$ does not divide $\# \Sha(E/K)$ if and only if \[[E(K) \otimes \Z_p: \Z_p (y_K \otimes 1)] = \prod_{\ell \mid N} c_\ell(E/\Q)^{(p)}.\] Under this condition, we extend the result of Bertolini \cite{Ber01} to cases where $p$ divides the product $\prod_{\ell \mid N} c_\ell(E/\Q) \cdot \prod_{v \mid p} d_v$, or when $\overline{\rho}_{E, p}: G_\Q \rightarrow \Aut(E_p)$ is not surjective.
	
	\begin{thm} \label{cor:main}
		Assume $N^{-} = 1$. Suppose $[E(K) \otimes \Z_p : \Z_p (y_K \otimes 1)] = \prod_{\ell \mid N} c_\ell(E/\Q)^{(p)}$ and that one of the following holds:
		\begin{enumerate}
			\item The prime $p$ is split in $\calO_K$.
			\item The prime $p$ is inert in $\calO_K$ and $d_v^{(p)} = 1$ for the unique prime $v \mid p$ in $K$.
		\end{enumerate}
		Then $H^1(\Gamma_\ac, Sel_p(E/K_\ac)) = 0$ and the module $X(E/K_\ac)$ does not have any non-zero finite submodules.
	\end{thm}
	
	\begin{proof}
		Since $\Sha(E/K)(p)$ is finite (see Remark \ref{rmk:Kol}), we may identify $E(K) \otimes \Z_p$ with $\calS_p(E/K)$. Note that the image of $\calS_p(E/K_\ac)_{\Gamma_\ac}$ in $\calS_p(E/K) = E(K) \otimes \Z_p$ contains $\Z_p \left(\prod_{v \mid p} d_v^{(p)}\right) (y_K \otimes 1)$ \cite[438]{P-R95}. Therefore $[\calS_p(E/K) : \calS_p(E/K_\ac)_{\Gamma_\ac}]$ divides \[\prod_{v \mid p} d_v^{(p)} \cdot [E(K) \otimes \Z_p : \Z_p (y_K \otimes 1)] = \prod_{\ell \mid N} c_\ell(E/\Q)^{(p)} \cdot \prod_{v \mid p} d_v^{(p)}.\] Once again, we note that this is precisely $(\# \ker(\gamma))^{1/2}$. Under our assumptions, $\#\ker(\delta)$ also divides the RHS by Lemma \ref{lem:diag}. It follows from Theorem \ref{thm:main} that $H^1(\Gamma_\ac, Sel_p(E/K_\ac)) = 0$  and the module $X(E/K_\ac)$ does not have any non-zero finite submodules.
	\end{proof}
	% \section{Pseudo-null submodules}
	
	We illustrate our results with numerical examples which fall under the hypotheses of Theorem \ref{cor:main} where Theorem \ref{thm:Ber01} does not apply. 
	
	\begin{ex} \label{ex:indef}
		Consider the elliptic curve $E/\Q$ with Cremona label $11a1$ and let $K = \Q(\sqrt{-19}), p = 5$. In this example, $5$ is split in $\calO_K$ and $[E(K) \otimes \Z_5 : \Z_5 (y_K \otimes 1)] =  c_{11}^{(5)} = 5.$ The Heegner point $y_K$ has trivial index in the free part of $E(K)$ and $[E(K) \otimes \Z_5: \Z_5(y_K \otimes 1)]$ is contributed by the $5$-torsion points $E_5(K).$		
		Other important properties of this example are the residual representation $\overline{\rho}_{E, 5}: G_\Q \rightarrow \Aut(E_5)$ is not surjective and $d_v^{(5)} = 5$ for each prime $v \mid 5$.
	\end{ex}
	
	\begin{ex} \label{ex:indef-2}
		Consider the elliptic curve $E/\Q$ with Cremona label $326b1$ and let $K = \Q(\sqrt{-31}), p = 5$. In this example, $5$ is also split in $\calO_K$ and $[E(K) \otimes \Z_5 : \Z_5 (y_K \otimes 1)] =  \prod_{\ell \mid N} c_{\ell}^{(5)} = 5.$ In contrast with the previous example, $E(K)$ has trivial torsion and the Heegner point $y_K$ has index $5$ in $E(K)$.
		The residual representation $\overline{\rho}_{E, 5}: G_\Q \rightarrow \Aut(E_5)$ is surjective, and $d_v^{(5)} = 1$ for each prime $v \mid 5$.
	\end{ex}
	
	We end this section by remarking that in the case where \[p \nmid \prod_{v \mid p} d_v \cdot \prod_{v \mid N} c_v(E/K)\] and $y(K) \notin pE(K)$, Matar and Nekov\'a\v r recently showed that $\Sel_p(E/K_\ac)$ is in fact cofree if we additionally assume that $E_p$ is an irreducible $G_\Q$-module \cite{MN19}. 
	
	\section{The vanishing of the $\mu$-invariant} \label{sec:mu=0}
	
	\subsection{The definite setting}
	We discuss some results in the literature regarding the vanishing of the $\mu$-invariant over the anticyclotomic extension. In the definite setting, we have the following result due to Pollack-Weston.
	
	\begin{thm} \label{PW11} \cite{PW11} Assume $\overline{\rho}_{E,p}: G_\Q \rightarrow \Aut(E_p)$ satisfies the following conditions:
		\begin{enumerate}
			\item $\overline{\rho}_{E,p}$ is surjective,
			\item if $q \mid N^{-}$ and $q \equiv \pm 1 \pmod{p},$ then $\overline{\rho}_{E,p}$ is ramified at $q$.
		\end{enumerate}
		Then $\mu(\Sel_p(E/K_\ac)) = 0$.
	\end{thm}
	
	The proof of this result uses the vanishing of the $\mu$-invariant of the $p$-adic $L$-function and the Iwasawa main conjecture.

	Recall that for a prime $v$ in $K$, we denote by $\widetilde{E}_v$ the reduction of $E$ modulo $v$. Define
	\[\widetilde{J_v}^1(K_\ac, E_p) = \begin{cases}
		\bigoplus_{w \mid v} H^1(K_{\ac, w}, (\widetilde{E}_v)_p) & \text{ if } v \mid p, \\
		\bigoplus_{w \mid v} H^1(K_{\ac, w}, E_p) & \text{ if } v \nmid p.
	\end{cases}\]
	Working with Galois cohomology of the residual representation $E_p$, we have the following statement in a previous work of the authors \cite{NS25}.
	
	\begin{thm} \label{thm:mu=0-def}
		Suppose that the Selmer group $\Sel_p(E/K_\ac)$ is cotorsion over $\Lambda(\Gamma_\ac)$ and the following hold:
		\begin{enumerate}
			\item $\overline{\rho}_{E,p}$ is surjective,
			\item $E(K_v)_p = 0$ for every prime $v$ lying above an inert prime in $\Q$.
		\end{enumerate}
		
		Then $\mu(\Sel_p(E/K_\ac)) = 0$ if and only if $H^2(K_S/K_\ac, E_p) = 0$ and the following map is surjective
		\[H^1(K_S/K_\ac, E_p) \rightarrow \bigoplus_{v \in S} \widetilde{J_v}^1(K_\ac, E_p).\]
	\end{thm}
	
	The condition $H^2(K_S/K_\ac, E_p) = 0$ is equivalent to the vanishing of the $\mu$-invariant of the dual fine Selmer group \cite[Theorem 3.5]{NS25}. For the cyclotomic extension $K_\cyc$, this vanishing condition is known as Conjecture A in the literature. We refer readers to \cite{CS05} where this conjecture was first proposed by John Coates and the second author.

	\subsection{The indefinite setting}
	
	In the indefinite setting, the following result is due to Hatley-Lei \cite{HL21}.
	\begin{thm} \cite{HL21} Assume the following conditions are satisfied:
		\begin{enumerate}
			\item $p \nmid 6 N \phi(N) \# \Pic(O_K)$,
			\item $a_p(E) ^2 \neq 1 \pmod{p}$,
			\item $\overline{\rho}_{E,p}: G_\Q \rightarrow \Aut(E_p)$ is absolutely irreducible.
		\end{enumerate}
		Then $\mu(\Sel_p(E/K_\ac)) = 0$.
	\end{thm}
	
	In the same vein as Theorem \ref{PW11}, the proof of this result also uses the Iwasawa main conjecture and the vanishing of the $\mu$-invariant of the BDP $p$-adic $L$-function established by Hsieh \cite{Hsi14} and later by Burungale \cite{Bur17} using different methods. We prove the vanishing of the $\mu$-invariant without assuming the Iwasawa main conjecture in Theorem \ref{thm:mu=0-main}.
	
	An important hypothesis in the aforementioned results is the irreducibility of $\overline{\rho}_{E,p}: G_\Q \rightarrow \Aut(E_p)$. It is not clear whether one should expect the $\mu$-invariant to vanish when $\overline{\rho}_{E,p}: G_\Q \rightarrow \Aut(E_p)$ is reducible. The necessity for this condition is suggested in the Introduction of \cite{Hsi14}. For the cyclotomic extension over $\Q$, it is known that the $\mu$-invariant can be positive when $\overline{\rho}_{E,p}: G_\Q \rightarrow \Aut(E_p)$ is reducible \cite[p.2]{GV00}.
	
	Recall the definition of the $p$-Selmer group $\Sel(E_p/L)$ in \ref{sec:prelim}, and let
	\[\Sel(E_p/K_\ac) := \varinjlim_{L} \Sel(E_p/L)\]
	where the limit is over the subextensions $L$ of $K_\ac$ that are finite over $K$. We also denote by $\Omega(\Gamma_\ac)$ the residual Iwasawa algebra $\Lambda(\Gamma_\ac)/p\Lambda(\Gamma_\ac)$, which is isomorphic to $\F_p \llbracket T \rrbracket$.
	
	Using the same techniques as Theorem \ref{thm:mu=0-def}, the vanishing of the $\mu$-invariant can also be phrased in terms of the residual representation $E_p$.

	\begin{thm} \label{thm:mu=0-indef}
		Suppose that the following hold:
		\begin{enumerate}
			\item $\overline{\rho}_{E,p}:  G_\Q \rightarrow \Aut(E_p)$ is surjective,
			\item $E(K_v)_p = 0$ for every prime $v$ lying above an inert prime in $\Q$,
			\item Conjecture A holds, i.e. the cohomology $H^2(K_S/K_\ac, E_p)$ vanishes.
		\end{enumerate}
		Then $\mu(\Sel_p(E/K_\ac)) = 0$ if and only if  and the cokernel of following map has $\Omega(\Gamma_\ac)$-corank $1$:
		\[H^1(K_S/K_\ac, E_p) \rightarrow \bigoplus_{v \in S} \widetilde{J_v}^1(K_\ac, E_p).\]
	\end{thm}
	
	Note that the second condition in Theorem \ref{thm:mu=0-indef} is redundant under the strong Heegner hypothesis, namely that every prime $\ell$ dividing the conductor $N$ of $E$ is split in $K/\Q$. 
	
	If we assume that the Heegner point $y_K \in E(K)$ has infinite order as in Section \ref{sec:indef}, we are able to provide conditions under which the $\mu$-invariant vanishes without Conjecture A using purely algebraic arguments.

	\begin{thm} \label{thm:mu=0-main} 
		Suppose that the Heegner point $y_K$ has infinite order and the following hypotheses hold:
		\begin{enumerate}
			\item $\overline{\rho}_{E,p}:  G_\Q \rightarrow \Aut(E_p)$ is surjective,
			\item $p \nmid \prod_{v \mid p} d_v \cdot \prod_{\ell \mid N^{+}} c_\ell(E/\Q).$
		\end{enumerate}
		Then the $\mu$-invariant of $\Sel_p(E/K_\ac)$ is trivial.
	\end{thm}
	
	\begin{proof}
		By the result of Kolyvagin \cite{Kol90}, we have $\dim_{\F_p} \Sel(E_p/K) = 1$. Analogous to Corollary \ref{cor:isom}, the control morphism
		\[\Sel(E_p/K) \rightarrow \Sel(E_p/K_\ac)^{\Gamma_\ac}\]
		
		is an isomorphism. It follows that
		\[\corank_{\Omega(\Gamma_\ac)} \Sel(E_p/K_\ac) \leq 1.\] 
		
		Recall that we denote by $(\Sel_p(E/K_\ac))_p$ the $p$-torsion part of the Selmer group $\Sel_p(E/K_\ac)$). The map
		\[\Sel(E_p/K_\ac) \rightarrow (\Sel_p(E/K_\ac))_p\]
		is an isomorphism since $H^1(K_S/K_\ac, E_p) \simeq H^1(K_S/K_\ac, E_{p^\infty})_p$ under our hypotheses (see \cite[Theorem 4.9]{NS25}). Hence, $X(E/K_\ac)/p X(E/K_\ac)$ has the same $\Omega(\Gamma_\ac)$-rank as the dual of $\Sel(E_p/K_\ac)$. We have seen that the dual of $\Sel(E_p/K_\ac)$ has $\Omega(\Gamma_\ac)$-$\corank$ at most $1$, so the same must hold for $X(E/K_\ac)/p X(E/K_\ac)$.
		
		On the other hand, the fact that $X(E/K_\ac)$ has $\Lambda(\Gamma_\ac)$-rank $1$ as discussed in \ref{sec:indef} implies that the $\Omega(\Gamma_\ac)$-rank of $X(E/K_\ac)/p X(E/K_\ac)$ is at least $1$. Thus the equality must occur, which implies that only the $\Lambda(\Gamma_\ac)$-rank of $X(E/K_\ac)$ contributes towards the $\Omega(\Gamma_\ac)$-rank of $X(E/K_\ac)/p X(E/K_\ac)$, and the $\mu$-invariant of $X(E/K_\ac)$ must be trivial.
	\end{proof}
	
	We provide the following numerical example in support of Theorem \ref{thm:mu=0-main}.
	\begin{ex}
		Let $K = \Q(\sqrt{-7})$, $p = 5$ and $E/\Q$ be the elliptic curve with Cremona label $1058d1$. The conductor $1058$ satisfies the strong Heegner hypothesis. In this example, $\bar{\rho}_{E, 5}$ is surjective, the Tamagawa product $\prod_{\ell \mid N} c_\ell(E/\Q) = 1$ and the Hasse invariant $a_5(E) = 2$, which implies $5 \nmid \prod_{v \mid 5} d_v$. 
	\end{ex}
	We have mentioned in Section \ref{sec:indef} that the work of Matar and Nekov\'a\v r \cite{MN19} implies $\Sel_p(E/K_\ac)$ is co-free and in particular has trivial $\mu$-invariant. It is natural to ask whether our example falls under their hypotheses. In our example, the Heegner point $y_K$ is divisible by $p$ in $E(K)$  whereas \cite{MN19} assumes the opposite, namely $y_K \not \in pE(K)$.
	
	\section*{Acknowledgements}
	The authors would like to thank Ashay Burungale, Antonio Lei and  Ahmed Matar for helpful conversations.

	%%  The bibliography
	
	%%  If your bibliography is in BibTeX format, use the following setup:
	%%  Style BST file for numbered citation:
	\bibliographystyle{plain}
	%%  Bibliography file (usually `*.bib')
	\bibliography{refs.bib}          
	%%
	%%  or include bibliography directly:
	%\begin{thebibliography}{9}
	%%  Use \bibitem{r1} or \bibitem[Surname(2010)]{r1} (for authoryear case)
	%%  Put author names in \textsc{} command in order to use small caps font
	%
	%\bibitem{}
	%\textsc{}
	%
	%\end{thebibliography}
	
\end{document}